\documentclass[a4paper,11pt]{article}
\usepackage{latexsym, amsmath, amssymb}
\usepackage{amsmath}
\usepackage{amsfonts}
\usepackage{amssymb}
\usepackage{amscd}
\usepackage{amsthm}
\usepackage{euscript,color}

\definecolor{red}{rgb}{1,0,0}
\theoremstyle{plain}
\newtheorem{theorem}{Theorem}[section]
\newtheorem{proposition}[theorem]{Proposition}
\newtheorem{remark}[theorem]{Remark}
\newtheorem{lemma}[theorem]{Lemma}
\newtheorem{definition}[theorem]{Definition}

\newtheorem{example}[theorem]{Example}

\newcommand{\HH}{\mathbb H}

\newcommand{\NN}{\mathbb N}

\newcommand{\PP}{\mathbb P}

\newcommand{\ZZ}{\mathbb Z}
\newcommand{\cA}{\mathcal A}
\newcommand{\cB}{\mathcal B}
\newcommand{\cC}{\mathcal C}

\newcommand{\cE}{\mathcal E}
\newcommand{\cF}{\mathcal F}
\newcommand{\cG}{\mathcal G}

\newcommand{\cI}{\mathcal I}

\newcommand{\cK}{\mathcal K}

\newcommand{\cO}{\mathcal O}
\newcommand{\cP}{\mathcal P}
\newcommand{\cQ}{\mathcal Q}

\newcommand{\imic}{\cong}

\newcommand{\To}{\longrightarrow}

\newcommand{\Proj}{\mathop{\null\mathrm {Proj}}\nolimits}

\newcommand{\Ext}{\mathop{\mathrm {Ext}}\nolimits}

\newcommand{\Hom}{\mathop{\mathrm {Hom}}\nolimits}

\newcommand{\hcf}{\mathop{\mathrm {hcf}}\nolimits}

\newcommand{\im}{\mathop{\mathrm {im}}\nolimits}
\newcommand{\lcm}{\mathop{\mathrm {lcm}}\nolimits}

\newcommand{\bdr}{\mathbf r}

\newcommand{\bdw}{\mathbf w}

\newcommand{\monad}[5]{{{#1} \overset{#2}{\lhook\joinrel\To}{#3}\overset{#4}{\relbar\joinrel\twoheadrightarrow}{#5}}}
\newcommand{\w}{w}
\newcommand{\Hst}{H_*}
\newcommand{\wreg}{\mathop{\null\mathrm {Wreg}}\nolimits}
\begin{document}
\title{Weighted Castelnuovo-Mumford Regularity and Weighted Global Generation}
\author{F.~Malaspina and G.K.~Sankaran
\vspace{6pt}\\
 {\small  Politecnico di Torino}\\
{\small\it  Corso Duca degli Abruzzi 24, 10129 Torino, Italy}\\
{\small\it e-mail: francesco.malaspina@polito.it}\\
\vspace{6pt}\\
{\small   University of Bath}\\
 {\small\it Bath BA2 7AY, England}\\
{\small\it e-mail: gks@maths.bath.ac.uk}}
\maketitle \def\thefootnote{}

\begin{abstract}
We introduce and study a notion of Castelnuovo-Mumford regularity
suitable for weighted projective spaces.
\end{abstract}
\section{Introduction}\label{sect:intro}
In chapter~14 of~\cite{Mumford} Mumford introduced the concept of regularity
for a coherent sheaf on a projective space $\PP^n$. It was soon clear
that it was a key notion and a fundamental tool in many areas of
algebraic geometry and commutative algebra.

From the algebraic geometry point of view, regularity measures the
complexity of a sheaf: the regularity of a coherent sheaf is an
integer that estimates the smallest twist for which the sheaf is
generated by its global sections. In Castelnuovo's much earlier
version, if $X$ is a closed subvariety of projective space and $H$ is
a general hyperplane, one uses linear systems (seen now as a precursor
of sheaf cohomology) to get information about $X$ from information
about the intersection of $X$ with $H$ plus other geometrical or
numerical assumptions on $X$.

From the computational and commutative algebra point of view, the
regularity is one of the most important invariants of a finitely
generated graded module over a polynomial ring. Roughly, it measures
the amount of computational resources that working with that module
requires. More precisely the regularity of a module bounds the largest
degree of the minimal generators and the degree of syzygies.

Extensions of this notion have been proposed over the years to handle
other ambient varieties instead of projective space: Grassmannians \cite{ArrondoMalaspina},
quadrics \cite{BallicoMalaspina1}, multiprojective spaces \cite{BallicoMalaspina2,CostaMiroRoig1}, $n$-dimensional smooth
projective varieties with an $n$-block collection~\cite{CostaMiroRoig1},
and abelian varieties \cite{PopescuPopa}.

In all these cases the ambient variety is smooth. Maclagan and Smith
\cite{MaclaganSmith} gave a variant of multigraded Castelnuovo-Mumford
regularity, motivated by toric geometry, which applies to some
singular varieties: for a different approach to multigraded
regularity, see \cite{SidmanVanTuyl}.

Since it often happens that a variety can be conveniently embedded in
a weighted projective space but embedding it in projective space
requires some arbitrary choices, or the use of many variables or high
degree equations, it is worthwhile to be able to import these ideas
into weighted projective spaces.

The first aim of this project is to introduce and study a notion of
regularity, and a related notion of globally generated sheaf, using
Koszul complexes, for weighted projective spaces. The theory of
\cite{MaclaganSmith} applies to weighted projective spaces, but as theirs is a
general theory for all toric varieties, we believe that it should be
possible to do better in this narrower context. In particular we want
the structure sheaf to be regular, which in general does not happen in
\cite{MaclaganSmith}. Specifically, the definitions in \cite{MaclaganSmith}, applied to
weighted projective spaces, take no account of the individual weights, and
the results are therefore only those that hold for all weighted
projective spaces (and more), irrespective of the weights.

\section{Generalities}\label{sect:general}
Fix a weighted projective space $\PP=\PP(\w_0,\dots,\w_n)$, which we
always write with the weights in decreasing order, $\w_0\geq\dots\geq
\w_n$.  There is a natural quotient map $\pi\colon\PP^n\to\PP$ (see
\cite[Thm.~3A.1]{BeltramettiRobbiano}).

We want to
follow~\cite{Canonaco} and regard $\PP$ as a graded scheme: see
Remark~\ref{rk:graded}. An alternative would be to
regard $\PP$ as a stack, as in
\cite[Example~7.27]{FantechiMannNironi}.
If we were
interested only in schemes (or varieties), we could assume that the
weights $(\w_0,\dots,\w_n)$ are reduced, i.e.\ no prime divides $n$ of
them, because every weighted projective space is isomorphic as a
scheme to a weighted projective space with reduced weights (see
\cite[Prop.~3C.5]{BeltramettiRobbiano}). If we were interested in
orbifolds we could similarly assume that $\hcf(w_0,\ldots,w_n)=1$.
However, the coordinate hyperplanes $\HH_j$ (see Lemma~\ref{lm:restr}
below) do not inherit these conditions, so we must continue to allow
arbitrary weights.

For a subset $I=\{\nu_1,\dots, \nu_s\}\subset \{0,\dots ,n\} $ with
$0\leq \nu_1< \dots< \nu_s\leq n$, we set $|\w_I|=\sum_{\nu\in I}
\w_{\nu}$. For convenience, we also write $\bdw_i$ for the sum of the
$i+1$ largest weights, $\bdw_i=\w_0+\ldots+\w_i=|\w_{0,\ldots,i}|$, and
we write $\bdw$ for the total weight, $\bdw=\bdw_n$.

As a $K$-scheme, $\PP(w_0,\ldots,w_n)=\Proj K[x_0,\ldots,x_n]$ with
grading given by $\deg x_i=w_i$. Accordingly, we can define the
twisting sheaf $\cO_\PP(1)$ and twists $\cE(t)$ for any coherent sheaf
$\cE$ and $t\in\ZZ$: see \cite{BeltramettiRobbiano}. In particular we
have $\omega_\PP=\cO_\PP(-\bdw)$.

\begin{remark}\label{rk:graded}
  We consider $\PP$ as a graded scheme, and the coherent sheaves are,
  from now on, to be understood as graded sheaves. The reason for this
  choice is that by working with graded sheaves on $\PP$ one avoids
  all the pathologies of $\PP$, because the sheaves of the form
  $\cO_{\PP}(j)$ are always invertible: see \cite[Chapter~2,
    p.~30]{Canonaco}.
\end{remark}

With this convention, for $\cE$ a coherent sheaf on $\PP$, we
have $\cE\otimes\cO_{\PP}(j)\cong \cE(j)$ for any integer $j$, so we
define the modules
\begin{equation}\label{eq:cohomologymodule}
H^i(\cE(\geq l))=\bigoplus_{t\geq l}H^i(\PP,\cE(t))
\end{equation}
and similarly $H^i(\cE(\leq l))$, for $l\in\ZZ$. We also use the
notation
\begin{equation}\label{eq:Histar}
  \Hst^i(\PP,\cE)=\bigoplus_{t\in\ZZ}H^i(\PP,\cE(t))=H^i(\cE(>-\infty)).
\end{equation}

Observe that $\pi_*\cO_{\PP^n}$ is a split vector bundle on $\PP$ by
\cite[Cor. 3A.2]{BeltramettiRobbiano}. More precisely
\begin{equation}\label{eq:pi*O}
\pi_*\cO_{\PP^n}\imic\bigoplus_{(r_0,\ldots,r_n)}\cO_\PP(-\bdr).
\end{equation}
where the sum is over all $(r_0,\ldots,r_n)$ such that $0\leq r_j<
\w_j$ for all $j=0,\dots, n$, and $\bdr=\sum_{j=0}^n r_j$.
\begin{lemma}\label{lm:vanishing}
For any $i\in\NN$, if $\cE$ is a vector bundle on $\PP$ then
$\Hst^i(\PP,\cE)=0$ if and only if $\Hst^i(\PP^n,\pi^*\cE)=0$.
\end{lemma}
\begin{proof}
Since $\pi$ is a finite morphism, we have
$\Hst^i(\PP^n,\pi^*\cE)\imic \Hst^i(\PP,\pi_*\pi^*\cE)$, and it is enough
to observe that (using the projection formula and \eqref{eq:pi*O})
\[
\pi_*\pi^*\cE\imic\cE\otimes\pi_*\cO_{\PP^n}\imic
\cE\otimes\bigg(\bigoplus_{(r_0,\ldots,r_n)}\cO_\PP(-\bdr)\bigg)\imic
\bigoplus_{(r_0,\ldots,r_n)}\cE(-\bdr),
\]
with notation as in~\eqref{eq:pi*O}.
\end{proof}

\section{Weighted Castelnuovo-Mumford Regularity}\label{sect:wregular}

We begin by recalling the usual definition of Castelnuovo-Mumford
regularity on projective space: see, for example,
\cite[Chapter~1.8]{Lazarsfeld}.
\begin{definition}\label{df:regular}
A coherent sheaf $\cF$ on $\PP^n$ is said to be \emph{$m$-regular}, for
$m\in\ZZ$, if
\[
H^i(\cF(m-i))=0
\]
for $i=1,\ldots,n$.
\end{definition}
It is well known (see \cite[Theorem~1.8.3]{Lazarsfeld}) that being
$m$-regular implies, in particular, that $H^0(\cF(m+1))\neq 0$, and in
fact much more than that: it is globally generated (and this even
holds for $\cF(m)$).

Maclagan and Smith in \cite{MaclaganSmith} gave a definition of
regularity for simplicial toric varieties. We refer to it as toric
regularity. On $\PP$ it reduces to the following
\begin{definition}\label{df:toricregular}
Let $\PP=\PP(\w_0,\dots,\w_n)$ and $k=\lcm(\w_0,\dots,\w_n)$.
A coherent sheaf $\cF$ on $\PP$ is said to be \emph{$m$-toric regular}
if, for $i=1,\dots ,n$
\[
H^i(\cF(m-ik))=0.
\]
\end{definition}
Here we have taken $\cC=\{\cO(k)\}$ in \cite[Definition~6.2]{MaclaganSmith}:
according to the definition of $\cC[i]$ given in
\cite[Section~4]{MaclaganSmith}, $\cF$ is toric $m$-regular if
$H^i(\cF(p))=0$ for all $i>0$ and for every $p\in m+(-i\cO(k)+\cC)$. That is,
$H^i(\cF(m-ik+t))=0$ for all $i>0$ and for every $t\in k\NN$, but it is enough to
consider $t=0$.

In our more restricted context, we want a definition that takes
account of the individual weights, which toric regularity does
not.

Our motivation for the definition we make comes from the Koszul complex.
On $\PP$ (with, as usual, $\w_0\geq\dots\geq \w_n$) we define
$\cA_j=\bigoplus\limits_{|I|=j+1}\cO(-|\w_I|)$. Then (see \cite[Lemma~2.1.3]{Canonaco}) there is a
Koszul complex on $\PP$ given by
\begin{equation}\label{eq:koszul}
0\To \cA_n\To \cA_{n-1}\To\dots\To \cA_0\To\cO\To 0.
\end{equation}
For example, if $\PP=\PP(5,3,2)$ then the Koszul complex is
\[
0\to \cO(-10)\to
\begin{array}{c}
\cO(-8)\\
\oplus\\
\cO(-7)\\
\oplus\\
\cO(-5)
\end{array}
\to
\begin{array}{c}
\cO(-5)\\
\oplus\\
\cO(-3)\\
\oplus\\
\cO(-2)
\end{array}
\to \cO\to 0.
\]

We give the following definition of weighted Castelnuovo-Mumford
regularity.
\begin{definition}\label{df:wregular}
Let $\PP=\PP(\w_0,\dots,\w_n)$ with $\w_0\geq\dots\geq \w_n$ and
$k=\lcm(\w_0,\dots,\w_n)$.  A coherent sheaf $\cF$ on $\PP$ is said to
be \emph{$m$-weighted regular}, which we abbreviate to
\emph{$m$-wregular}, if for $i=1,\dots ,n$
\[
H^i(\cF(t+(m+1)k-\bdw_i))=0
\]
for every $t\geq 0$, and also
\[
H^0(\cF((m+1)k))\neq 0.
\]
We often write \emph{wregular} to mean $0$-wregular.\\
We define the \emph{wregularity} of $\cF$, $\wreg(\cF)$, as the smallest
integer $m$ such that $\cF$ is $m$-wregular.
\end{definition}

\begin{remark}\label{rk:weight1}
For $\PP=\PP^n$, wregularity and toric regularity both coincide with
the usual notion of Castelnuovo-Mumford regularity.
\end{remark}
Indeed, in this case we have $m=0$ and $\w_0=\dots=\w_n=1$, so $k=1$
and $\w_i=i+1$, so taking $t=0$ in Definition~\ref{df:wregular} we get
\[
H^i(\cF(k-\bdw_i))= H^i(\cF(-i)).
\]

\begin{lemma}\label{lm:Owregular}
For $\PP$ any weighted projective space, $\wreg(\cO_\PP)=0$.
\end{lemma}
\begin{proof}
In fact for any $t\in\ZZ$ we have $H^i(\cO_\PP(t))=0$ for $0<i<n$, and
$H^0(\cO_\PP(k))\neq 0$.
For $0$-wregularity we also need $H^n(\cO(k-\bdw))=0$, but this holds
because $H^n(\cO(k-\bdw))$ is Serre dual to $H^0(\cO(-k))$, which is
zero. (See \cite[Section~6B]{BeltramettiRobbiano} and \cite[Proposition~2.1.4]{Canonaco} for
Serre duality in this context.)

However, $\cO$ is not $-1$-wregular because
$H^n(\cO(-\bdw))\cong H^0(\cO)\neq 0$.
\end{proof}

On the other hand we cannot expect $\cO_\PP$ to be toric regular for
arbitrary weights. In fact $H^n(\cO(-nk))\cong H^0(\cO(nk-\bdw))$ which
is non-zero in general. However, $\cO(nk)$ is always toric regular.

A significant difference between Definition~\ref{df:wregular} and
Definitions~\ref{df:regular} and \ref{df:toricregular} is that we have
imposed a non-vanishing condition, because we lack a counterpart to
Mumford's theorem~\cite[Theorem~1.8.3]{Lazarsfeld}: see
Example~\ref{ex:nonvanish} below. With this in mind, we make the
following definition.

\begin{definition}\label{df:semiwregular}
A coherent sheaf $\cF$ on $\PP$ is said to
be \emph{$m$-semiwregular} if for $i=1,\dots ,n$
\[
H^i(\cF(t+(m+1)k-\bdw_i))=0
\]
for every $t\geq 0$.
\end{definition}

\begin{example}\label{ex:nonvanish}
If $\PP=\PP(3,2)$ then $\cO_\PP(-5)$ is $0$-semiwregular but not
$0$-wregular, whereas $\cO_\PP(-4)$ is $0$-wregular.
\end{example}
In fact, $m=0$ and $k=6$, so $H^1(\cO(-5+6-5))\cong H^0(\cO(-1))=0$,
and thus if we take $\cF=\cO(-5)$ then the condition
$H^1(\cO(-5)\otimes\cO(t+6-(3+2)))=0$ is satisfied for every $t\geq
0$, but $H^0(\cO(-5)\otimes\cO(k))=H^0(\cO(1))=0$.

On the other hand, for $\cO(-4)$ the condition
$H^1(\cO(-4)\otimes\cO(t+6-(3+2)))=0$ is satisfied for every $t\geq 0$,
but $H^0(\cO(-4)\otimes\cO(k))=H^0(\cO(2))\neq 0$. So $\cF(1)=\cO(-4)$
is wregular.
\smallskip

Now we show how weighted regularity and weighted semiregularity behave under
pullback along the natural covering map from $\PP^n$.

\begin{lemma}\label{lm:pbwreg}
Let $\cF$ be an $m$-semiwregular (or $m$-wregular) coherent sheaf on
$\PP$.  Then $\pi^*\cF$ is $((m+1)k-n+\bdw-\bdw_1)$-regular on $\PP^n$.
\end{lemma}

\begin{proof}
We want to show that, for $q=(m+1)k-n+\bdw-\bdw_1$ and for any
$i=1,\dots,n$, we have
\[
h^i(\PP^n, \pi^* \cF(q-i))=0.
\]
By~\eqref{eq:pi*O}, $\pi_*\pi^* \cF\cong \bigoplus_{(r_0,\ldots,r_n)}\cF(-\bdr)$,
where $0\leq r_j< \w_j$ for all $j=0,\dots, n$.  The smallest twist that occurs is $-\bdw+n+1$ (when $r_j=w_j-1$ for every $j$). Since
\[
H^i(\PP^n,\pi^* \cF(q-i))\imic H^i(\PP,\pi_*\pi^*\cF(q-i))
\]
it is enough to show that $H^i(\PP,\cF(q-i-\bdw+n+1))$ vanishes, for each
$i$, with $q$ as above.

If $i=1$ and $q=(m+1)k-n+\bdw-\bdw_1$ we get
\[
H^1(\PP,\cF(q-1-\bdw+n+1))=H^1(\cF((m+1)k-\bdw_1))
\]
which is zero because $\cF$ is $m$-wregular. Hence
\[
H^1(\PP,\pi_*\pi^*\cF((m+1)k-n+\bdw-\bdw_1-1))=0.
\]

For $2\le i\le n$ we have
\[
\Hst^i(\PP,\cF(t+(m+1)k-\bdw_i))=0
\]
for every $t\ge 0$, by $m$-wregularity. However, $\bdw_i \ge \bdw_1+i-1$,
since the weights are positive integers, so
\begin{align*}
q-i-\bdw+n+1 &= (m+1)k-n+\bdw-\bdw_1-i-\bdw+n+1\\
&= (m+1)k-\bdw_1-i+1\\
&\ge m(k+1)-\bdw_i
\end{align*}
and hence $H^i(\PP,\cF(q-i-\bdw+n+1))=0$, as required.
\end{proof}

Next we show how weighted semiregularity behaves under restriction to
coordinate hyperplanes. Weighted regularity does not behave well, in
general, as Example~\ref{ex:badwreg} illustrates.

\begin{lemma}\label{lm:restr}
Suppose that $\cF$ is an $m$-semiwregular coherent sheaf on $\PP$ and
let $\HH_j=\{x_j=0\}$ be the $j$-th coordinate hyperplane. We put
$k_j=\lcm(\w_0,\dots , \w_{j-1}, \w_{j+1},\dots , \w_n)$ and
$z_j=k/k_j$. Then $\cF_{\HH_j}=\cF\otimes \cI_{\HH_j}$ is
$((m+1)z_j-1)$-semiwregular on $\HH_j$.
\end{lemma}
\begin{proof}
We consider $\HH_j\imic \PP(\w_0,\dots , \w_{j-1}, \w_{j+1},\dots ,
\w_n)$: note that the sum of the first $i+1$ of these weights is
$\bdw_i$ if $i<j$ and is $\bdw_{i+1}-\w_j$ if $i\ge j$. Thus, writing
$\cE=\cF_{\HH_j}$, we want to show that, for $q=(m+1)z_j-1$ and for
any $i=1,\dots , n-1$
\begin{align}
H^i(\HH_j, \cE(t+(q+1)k_j-\bdw_i))&= 0 && \text{if $i<j$}\label{eq:wregrestricted<}\\
H^i(\HH_j, \cE(t+(q+1)k_j-(\bdw_{i+1}-\w_j)))&= 0 && \text{if
  $i\ge j$},\label{eq:wregrestricted>}
\end{align}
for every integer $t\geq 0$.

Let us consider the exact sequence
\begin{equation}\label{eq:FEseq}
0\To\cF(-\w_j)\To\cF\To\cE\To 0,
\end{equation}
coming from tensoring $\cF$ with the restriction sequence
\[
0\To\cO_\PP(-\w_j)\To\cO_\PP\To\cO_{\HH_j}\To 0.
\]
If $i<j$ we twist \eqref{eq:FEseq} by $t+(m+1)k-\bdw_i$: in cohomology,
this gives (for $0<i<n$ and $t\ge 0$)
\begin{multline}\label{eq:long_i<j}
H^i(\PP,\cF(t+(m+1)k-\bdw_i))\To H^i(\HH_j,\cE(t+(m+1)k-\bdw_i))\\
\To H^{i+1}(\PP,\cF(t+(m+1)k-\bdw_i-w_j)).
\end{multline}
The first of these terms vanishes because $\cF$ is $m$-wregular, and
the $m$-wregularity also gives $H^{i+1}(\PP,\cF(t'+(m+1)k-\bdw_{i+1})=0$ for any
$t'\ge 0$. In particular, since $i+1\le j$ we have $\w_{i+1}\ge \w_j$
and we may take $t'=t+\w_{i+1}-\w_j$, giving us vanishing of the third
term in \eqref{eq:long_i<j}. Thus the middle term also vanishes, and
since $(q+1)k_j=(m+1)k$ that proves~\eqref{eq:wregrestricted<}.

The proof for the second case, $i\ge j$, is similar. This time we
twist \eqref{eq:FEseq} by $t+(m+1)k-(\bdw_{i+1}-\w_j)$. In cohomology
this gives
\begin{multline}\label{eq:long_i>j}
H^i(\PP, \cF(t+(m+1)k-(\bdw_{i+1}-\w_j))\To H^i(\HH_j,
\cE(t+(m+1)k-(\bdw_{i+1}-\w_j)))\\
\To H^{i+1}(\PP, \cF(t+(m+1)k-\bdw_{i+1})).
\end{multline}
The third of these terms vanishes because $\cF$ is $m$-wregular, and
the $m$-wregularity also gives $H^i(\PP,\cF(t'+(m+1)k-\bdw_i)=0$ for
any $t'>0$. Now since $i+1\ge j$ we have $\w_{i+1}\le \w_j$ and we may
take $t'=t-\w_{i+1}+\w_j$, giving us vanishing of the first term in
\eqref{eq:long_i>j}. Again, the middle term also vanishes and this
proves~\eqref{eq:wregrestricted>}.
\end{proof}

\begin{example}\label{ex:badwreg}
If $\PP=\PP(3,2,1)$, then $z_2=1$, it is easy to check that $\cF=\cO_\PP(-5)$ is
$0$-wregular. In fact, $k=6$, so $H^2(\cO(-5+6-6))\cong H^0(\cO(-1))=0$,
and thus if we take $\cF=\cO(-5)$ then the condition
$H^2(\cO(-5)\otimes\cO(t+6-(3+2+1)))=0$ is satisfied for every $t\geq
0$, and $H^0(\cO(-5)\otimes\cO(k))=H^0(\cO(1))\not=0$.
However $\cF_{\HH_2}$ is not $0$-wregular by
Example~\ref{ex:nonvanish}.
\end{example}

\begin{remark}\label{rk:restrictweight1}
Let $\cF$ be an $m$-wregular coherent sheaf on $\PP$. If
$\w_0=\dots=\w_n=1$, then $(m+1)k-n+\bdw_{n-2}=m$, and $\pi^* \cF$ is
$m$-regular. More generally, if $\w_j=1$, then $z_j(m+1)-1=m$ and
$\cF_{\HH_j}$ is $m$-semiwregular on~$\HH_j$.
\end{remark}
We cannot expect the above properties for toric regularity.
\smallskip

We can give a notion of global generation adapted to this weighted
situation.

\begin{definition}\label{df:wgg}
A  coherent sheaf $\cF$ on $\PP$ is said to be \emph{weighted globally
  generated} (abbreviated to \emph{wgg}) if, for any $x\in \PP$, the
map
\[
\mu\colon \bigoplus_{j=0}^n H^0(\cF(k-\w_j))\otimes\cO_x\to \cF_x(k),
\]
where $\mu(\sum_{j=0}^n f_j\otimes e_x)=\sum_{j=0}^n f_jx_je_x$, is
surjective.
\end{definition}
This reduces to the usual definition of globally generated in the case
of $\PP^n$, when $\w_0=\dots=\w_n=1$ and $k=1$, so $\bigoplus_{j=0}^n
H^0(\cF(k-\w_j))\cong H^0(\cF)\otimes H^0(\cO(1))$. In fact we have a
surjection
\[
\mu_x\colon H^0(\cF)\otimes H^0(\cO(1))\otimes\cO_x\to \cF_x(1)
\]
and a surjection
\[
H^0(\cF)\otimes H^0(\cO(1))\otimes\cO_x\to H^0(\cF)\otimes \cO_x(1).
\]
So we may construct a surjection
\[
H^0(\cF)\otimes \cO_x(1)\to \cF_x(1).
\]
Finally we may identify $\cO_{\PP}(1)$ with $\cO_{\PP}$ at $x$ and $\cF(1)$ with $\cF$ at $x$. So we may conclude that $\cF$ is globally generated.
\begin{proposition}\label{pr:Owgg}
$\cO_\PP$ is wgg.
\end{proposition}
\begin{proof}
We want to show that for any $i=0,\dots, n$ the map
\[
\mu_i\colon \bigoplus_{j=0}^n H^0(\cO(k-\w_j))\otimes \Gamma
(D_+(x_i),\cO)\to \Gamma (D_+(x_i),\cO(k))
\]
is surjective, where $D_+(x_i)$ denotes as usual the locus $(x_i\neq
0)$.

Let $u\in\Gamma (D_+(x_i),\cO(k))$; then $u=ax_i^{-s}$ with $s>0$ and
$a$ a monomial of degree $s\w_i+k$. Therefore $u=ax_i^{sk-s}/x_i^{sk}$
and
\[
\deg (ax_i^{sk-s})=s\w_i+k+\w_i(sk-s)=\w_isk+k=k(\w_is+1).
\]
So
$ax_i^{sk-s}$ is a monomial containing $x_i^{sk-s}$ and its degree is
a multiple of $k$.  This means that we can write $ax_i^{sk-s}=a'b$
where $a'=x_i^{k/\w_i-1}$ and has degree $k-\w_i$; then $b/x_i^{sk}$ has
degree $k(s\w_i+1)-k+\w_i-sk\w_i=\w_i$ so the map
\[
H^0(\cO(k-\w_i))\otimes \Gamma (D_+(x_i),\cO(\w_i))\To \Gamma
(D_+(x_i),\cO(k))
\]
is surjective.

Finally let us notice that $\Gamma (D_+(x_i),\cO(\w_i))\cong \Gamma
(D_+(x_i),\cO)$.

This means that the map $\mu_i$ is surjective for any $i=0,\dots n$,
and hence the map $\mu$ of Definition~\ref{df:wgg} is surjective at every point $x\in\PP$.
\end{proof}

In general, $\cO_\PP(m)$ is not globally generated in the usual sense: see
for example \cite[Theorem~4B.7]{BeltramettiRobbiano}. On the other hand we have the
following proposition.
\begin{proposition}\label{pr:ggimplieswgg}
If $m>0$ and $\cO_\PP(m)$ is globally generated then $\cO_\PP(m)$ is wgg.
\end{proposition}
\begin{proof} We want to show that for any $i=0,\dots, n$ the map
\[
\mu_i\colon\bigoplus_{j=0}^n H^0(\cO(k+m-\w_j))\otimes \Gamma
(D_+(x_i),\cO)\To \Gamma (D_+(x_i),\cO(m+k))
\]
is surjective.

Let $u\in\Gamma (D_+(x_i),\cO(m))$: then $u=ax_i^{-s}$ with $s>0$ and
$a$ a monomial of degree $s\w_i+m$. Therefore $u=ax_i^{sk-s}/x_i^{sk}$
and
\[
\deg (ax_i^{sk-s})=s\w_i+m+\w_i(sk-s)=\w_isk+m
\]
Now since
$\cO(m)$ is globally generated we can write $ax_i^{sk-s}=a'b$, where
$a'$ has degree~$m$ and $b/x_i^{sk}$ has degree~$0$. This means that
we can write $a'b=a''b'$ where $a''=a'x_i^{-1}$ and has degree
$m-\w_i$, and $b'=bx_i$ so that $b'/x_i^{sk}$ has degree $\w_i$. In
this way we have that the map
\[
H^0(\cO(k+m-\w_i))\otimes \Gamma (D_+(x_i),\cO(\w_i))\To \Gamma
(D_+(x_i),\cO(m+k))
\]
is surjective.
\end{proof}

Now we prove the analogues for weighted regularity of the main
properties of Castelnuovo-Mumford regularity.
\begin{theorem}\label{th:wregproperties}
Let $\cF$ be a wregular coherent sheaf on $\PP$.
  \begin{enumerate}
  \item[(i)] $H^0(\cF(k))$ is spanned by $H^0(\cF(k-\w_0))\oplus\dots\oplus H^0(\cF(k-\w_n))$.
  \item[(ii)] $\cF$ is $m$-wregular for all $m\geq 0$.
  \item[(iii)] $\cF$ is wgg.
  \end{enumerate}
\end{theorem}
\begin{proof}
(i) is clear from the Koszul sequence~\eqref{eq:koszul} twisted by
  $\cF(k)$. Moreover since $H^0(\cF(k))\neq 0$ the surjection is
  non-trivial.\\
(ii) is clear by the definition of wregularity.\\
(iii) we prove as follows. Choose $l\in k\ZZ$ so that $\cF(k+l)$ and
  $\cO(l)$ are globally generated, which holds for $l\gg 0$, and
  consider the (not exact!) sequence
\begin{multline*}
\bigoplus_j H^0(\cF(k-\w_j))\otimes H^0(\cO(l))\otimes \cO
\overset{\mu}\To
H^0(\cF(k))\otimes H^0(\cO(l))\otimes\cO\\
\overset{\mu'}\To
H^0(\cF(k+l))\otimes\cO
\overset{\mu''}{\To}
\cF(k+l).
\end{multline*}
Notice that $\mu$ is non-trivial and surjective by~(i), and $\mu'$ and
$\mu''$ are both surjective because $\cF(k+l)$ and $\cO(l)$ are
globally generated. Near a point $x$, fix an isomorphism between
$\cO(l)$ and $\cO$: this identifies $\cO(l)$ with $\cO$ and $\cF(k)_x$
with $\cF(k+l)_x$. Then $H^0(\cO(l))$ becomes just a vector space of
elements of the local ring $\cO_x$, so we have that $\cF(k)$ is wgg.
\end{proof}

\section{Monads on weighted projective spaces}\label{sect:monads}
In this section we assume that $n=\dim\PP\geq 3$. We begin with a
preliminary definition.
\begin{definition}\label{df:minimalmap}
Suppose that $\cE$ and $\cE'$ are vector bundles on a projective
variety $X$. A surjective map $\eta\colon \cE\to \cE'$ is said to be
\emph{minimal} if no rank~$1$ direct summand of $\cE'$ is the image of
a line subbundle of~$\cE$.
\end{definition}
Next we recall the basic definitions about monads, due to
Horrocks~\cite{Horrocks}.
\begin{definition}\label{df:monad}
A sequence of bundles on a projective variety $X$
\[
\monad{\cA}{\alpha}{\cB}{\beta}{\cC}
\]
such that $\cA$ and $\cC$ are sums of line bundles, $\alpha$ is
injective, $\beta$ is surjective and $\beta\alpha =0$ is called a
\emph{monad} on $X$.\\
The vector bundle $\cE=\frac{\ker\beta}{\im \alpha}$ is
called the \emph{homology} of the monad.\\
A monad is said to be \emph{minimal} if the maps $\alpha^\vee\colon
\cB^\vee\to \cA^\vee$ and
$\beta\colon \cB\to\cC$ are minimal.
\end{definition}
In particular if $\cB$ is a sum of line bundles, the maps $\alpha$ and
$\beta$ are just matrices and then minimal means that no matrix entry
is a non-zero scalar both in $\alpha$ and in $\beta$.

Horrocks showed in~\cite{Horrocks} that every bundle $\cE$ on $\PP^n$
with $n\geq 3$ is the homology of a minimal monad. Now we extend this
correspondence to $\PP$: we generalize the proof of Proposition $3$ in \cite{BH}. First we need a definition (see
equation~\eqref{eq:cohomologymodule} in
Section~\ref{sect:general} for the notation).
\begin{definition}\label{df:minimal_lres}
For $l\in\ZZ$, a \emph{minimal $l$-resolution} of a bundle $\cE$ is an
exact sequence
\[
0\To \cE \To \cP  \stackrel{\pi}\To  \cC \To 0
\]
in which $\cC$ splits, $\pi$ is minimal and $H^1(\cP(\geq l))=0$.
\end{definition}
\begin{theorem}\label{thm:bundlesmonad}
Every bundle $\cE$ on $\PP$ is the homology of a minimal monad with
$\cB$ satisfying
\begin{enumerate}
\item[(i)] $\Hst^1(\cB)=\Hst^{n-1}(\cB)=0$
\item[(ii)] $\Hst^i(\cB)=\Hst^i(\cE)$ if $1<i<n-1$.
\end{enumerate}
\end{theorem}
\begin{proof}
The module $\Hst^i(\cE)$ has finite length for $0<i<n$, because $\PP$
is arithmetically Cohen-Macaulay and subcanonical, and for any $t\in
\ZZ$ we have
\[
H^i(X, \cE(t))\imic H^{n-i}(X,\cE^\vee(-t+e))^\vee.
\]
We start by proving that every bundle $\cE$ has a minimal
$l$-resolution for each $l\in\ZZ$. Consider the module $H^1(\cE(\geq
l))$ and a minimal system of generators $g_1,\ldots, g_r$. For each
$i\ge l$, write $q_i$ for the number of generators in degree $i$:
since the module has finite length there is an $l_0$ such that $q_i
=0$ if $i>l_0$.  So our system $g_1,\ldots, g_r$ is an element of
\[
q_l H^1(\cE(l))\oplus\dots\oplus q_{l_0} H^1(\cE(l_0))\imic
H^1(\cE\otimes \cC^{\vee})
\]
where $\cC=q_l \cO(-l)\oplus\dots\oplus q_{l_0} \cO(-l_0)$.
But, since $H^1(\cE\otimes \cC^\vee)\imic \Ext^1 (\cC, \cE)$ we have
$\{g_i\}\in \Ext^1 (\cC, \cE)$, so we can associate an extension
\begin{equation}\label{eq:lres}
0\To \cE \To \cP  \To  \cC \To 0
\end{equation}
to our system of generators.

Now, looking at the sequence in cohomology we see that the map
\[
f\colon H^0(\cC(\geq l))\To H^1((\cE(\geq l))
\]
is surjective by construction.  Moreover, since $\PP$ is ACM, all the
intermediate cohomology of $\cC$ vanishes and we can conclude that
$H^1(\cP(\geq l))=0$. Thus the sequence \eqref{eq:lres} is an
$l$-resolution, and it is minimal because the system $\{g_i\}$ was
chosen minimal.

This $l$-resolution will be the last column of our display.

In the same way, for every $l'\in \ZZ$, we can find an $l'$-resolution
\begin{equation}\label{eq:dualres}
0\To \cE^\vee \To \cQ^\vee  \To  \cA^\vee \To 0
\end{equation}
for $\cE^\vee$ and the dual of~\eqref{eq:dualres} will be the first
row, so we have
\[
\begin{array}{ccccccccc}
 & & & & & &0 & & \\
 & & & & & &\downarrow & & \\
0& \to &\cA &\to &\cQ &\to &\cE &\to &0\\
 & & & & & &\downarrow & & \\
 & & & & & &\cP & & \\
 & & & & & &\downarrow & & \\
 & & & & & &\cC & & \\
 & & & & & &\downarrow & & \\
 & & & & & &0 & & \\
\end{array}
\]
where $\cA=\bigoplus\nolimits_i\cO(a_i)$ is a bundle without
intermediate cohomology, with $a_i\geq l'$ for all~$i$. Moreover
\[
H^{n-1}(\cQ(\leq -l'+e))\cong H^1(\cQ^\vee(\geq l'))=0.
\]
We observe that if $n\geq 3$
\[
\Ext^i(\cC, \cA) = H^i(\cC^\vee\otimes \cA)=0
\]
for $i=1,\, 2$. Then applying the functor $\Hom(\bullet ,\cA)$
to~\eqref{eq:lres} we have
\[
0=\Ext^1(\cC, \cA) \to \Ext^1(\cP, \cA)\to \Ext^1(\cE, \cA) \to
\Ext^2(\cC, \cA)=0,
\]
so
\[
\Ext^1(\cP, \cA)\imic \Ext^1(\cE, \cA).
\]
This means that the extension in our row (the dual of
\eqref{eq:dualres}) comes from the unique extension
\[
0 \To \cA \overset{\alpha}\To \cB \To \cP \To 0
\]
and we have
\[
\begin{array}{ccccccccc}
 & & & &0 & &0 & & \\
 & & & &\downarrow & &\downarrow & \\
0&\to &\cA &\to &\cQ &\to &\cE &\to &0 \\
 & &\|& &\downarrow & &\downarrow & & \\
0&\to &\cA &\xrightarrow{\alpha} &\cB &\to &\cP &\to &0\\
 & & & &\downarrow & &\downarrow & & \\
 & & & &\cC &= &\cC & & \\
 & & & &\downarrow & &\downarrow & & \\
 & & & &0 & &0 & & \\
 \end{array}.
\]
This is the display of the monad
\[
\monad{\cA}{\alpha}{\cB}{\beta}{\cC}.
\]
The minimality comes from the minimality of the two resolutions.  From
the first row we see that $\Hst^i(\cE)\cong \Hst^i(\cQ)$ for
$0<i<n-1$.  Looking at the first column in cohomology,
\[
0=H^{n-1}(\cQ(\leq -l'+e))\to H^{n-1}(\cB(\leq -l'+e)) \to
H^{n-1}(\cC(\leq -l'+e))=0,
\]
we have that $H^{n-1}(\cB(\leq -l'+e))=0$, and $\Hst^i(\cE)\cong
\Hst^i(\cQ)\cong \Hst^i(\cB)$ for $1<i<n-1$.

Looking at the second row in cohomology,
\[
0=H^1(\cA(\geq l))\to H^{n-1}(\cB(\geq l)) \to H^{n-1}(\cC(\geq l))=0,
\]
we see that $H^{1}(\cB(\geq l))=0$.  If we choose $l$ and $l'$ small
enough we get the claimed conditions~(i) and~(ii) above:
\begin{enumerate}
\item[(i)] $\Hst^1(\cB)=\Hst^{n-1}(\cB)=0$
\item[(ii)] $\Hst^i(\cB)=\Hst^i(\cE)$ for $1<i<n-1$.
\end{enumerate}
\end{proof}

If the bundles in the monad all split, we can get some results about
wregularity for $\cE$.

\begin{definition}\label{df:quasilinear}
A  monad on $\PP$  is called \emph{quasi-linear} if it has the form
\begin{equation}\label{eq:quasilinear}
\monad{\bigoplus_{i=1}^s\cO_\PP(a_i)}{\alpha}{\bigoplus_{l=1}^{r+s+t}\cO_\PP(b_l)}{\beta}{\bigoplus_{j=1}^t\cO_\PP(c_j)}.
\end{equation}
\end{definition}
By convention, we shall write the twists in increasing order, $a_i\le
a_{i+1}$, etc.: note that this is the opposite of our convention for
the weights.

We prove an analogue of \cite[Theorem~3.2]{CostaMiroRoig2}.

\begin{theorem}\label{thm:monadwreg}
Let $\cE$ be a rank~$r$ vector bundle on $\PP$ which is the homology
of a quasi-linear monad~\eqref{eq:quasilinear}. Put $c=\sum_{j=1}^t
c_t$. Then $\cE$ is $m$-wregular for any integer $m$ such that
$H^0(\cE((m+1)k))\neq 0$ and
\begin{equation}\label{eq:monadwreg}
(m+1)k\geq\max\{(n-1)c_t-(b_1+\dots+b_{t+n})-(\bdw-\bdw_1)+1+c,
-b_1+1,-a_1+1\}.
\end{equation}
\end{theorem}
\begin{proof}
Let us consider the short exact sequences from the display of the
monad:
\[
0\to
\cK\to\bigoplus_{l=1}^{r+s+t}\cO_\PP(b_l)\xrightarrow{\beta}\bigoplus_{j=1}^t\cO_\PP(c_j)\to
0
\]
and
\[
0\to\bigoplus_{i=1}^s\cO_\PP(a_i)\to \cK\to \cE\to 0.
\]
We get $H^i(\cK(p))=H^i(\cE(p))=0$ for any integer $p$ and any
$i=2,\dots , n-2$. Moreover if $p\geq\max\{-b_1-\bdw+1,-a_1-\bdw+1\}$ we
have also $H^i(\cK(p))=H^i(\cE(p))=0$ for $i\geq n-1$. So if
$(m+1)k\geq\max\{-b_1+1,-a_1+1\}$ we may conclude that
\[
H^n(\cE((m+1)k-\bdw)=0.
\]
To see which are the $p$ for which $H^1(\cK(p))\cong H^1(\cE(p))=0$,
we consider the Buchsbaum-Rim complex associated to
\[
\cF=\bigoplus_{l=1}^{r+s+t}\cO_\PP(b_l)
\overset{\beta}{\relbar\joinrel\twoheadrightarrow}
\cG=\bigoplus_{j=1}^t\cO_\PP(c_j),
\]
which is the complex
\begin{multline}\label{eq:buchsbaumrim}
S^{r+s-1}\cG^\vee\otimes\wedge^{r+s+t}\cF \to
S^{r+s-2}\cG^\vee\otimes\wedge^{r+s+t-1}\cF \to \dots \to
S^{2}\cG^\vee\otimes\wedge^{3+t}\cF\\
\to \cG^\vee\otimes\wedge^{2+t}\cF
\to \wedge^{1+t}\cF\to \cF\otimes\cO_\PP(c)\to \cG\otimes\cO_\PP(c)\to
0.
\end{multline}
We cut \eqref{eq:buchsbaumrim} into short exact sequences
\begin{align*}
 0&\to \cK\otimes\cO_\PP(c)\to \cF\otimes\cO_\PP(c)\to \cG\otimes\cO_\PP(c)\to 0,\\
 0&\to \cK_2\to \wedge^{1+t}\cF\to \cK\otimes\cO_\PP(c)\to 0,\\
 0&\to \cK_3\to \cG^\vee\otimes\wedge^{2+t}\cF\to \cK_2\to 0,\\
 &\qquad\qquad\vdots\\
 0&\to \cK_n\to S^{n-2}\cG^\vee\otimes\wedge^{t+n-1}\cF\to \cK_{n-1}\to 0,\\
 0&\to \cK_{n+1}\to S^{n-1}\cG^\vee\otimes\wedge^{t+n}\cF\to \cK_n\to 0.
\end{align*}
Note that
\[
S^{n-1}\cG^\vee\otimes\wedge^{t+n}\cF=\bigoplus_q\cO_\PP(d_q)
\]
where $d_q=(b_{l_1}+\dots +b_{l_{t+n}})-(c_{j_1}+\dots +c_{j_{n-1}})$
with $l_1<\dots < l_{t+n}$ and $j_1\leq\dots \leq j_{n-1}$.  Now from
the cohomological exact sequences associated to the above short exact
sequences tensored by $\cO_\PP(p-c)$ we get
\begin{multline}
h^1(\cK(p))=h^2(\cK_2(p-c))=\dots =h^n(\cK_n(p-c))\\
\qquad\leq h^n(S^{n-1}\cG^\vee\otimes\wedge^{t+n}\cF\otimes\cO_\PP(p-c))
\end{multline}
which is zero if $p\geq (n-1)c_t-(b_1+\dots +b_{t+n})-\bdw+1+c$. In
fact, since
\[
(n-1)c_t-(c_{j_1}+\dots +c_{j_{n-1}})\geq 0
\]
and
\[
(b_{l_1}+\dots +b_{l_{t+n}})-(b_1+\dots +b_{t+n})\geq 0,
\]
we have $d_q+p-c\geq -\bdw$. So we get
\[
H^1(\cE((m+1)k-(\w_{n-1}+\w_n)))=0
\]
if $(m+1)k\geq (n-1)c_t-(b_1+\dots+b_{t+n})-(\w_0+\dots +\w_{n-2})+1+c$.
\end{proof}
\begin{remark}\label{rk:unweighted}
In the case of $\PP=\PP^n$ the bound~\eqref{eq:monadwreg} reduces to
\[
m+1\geq\max\{(n-1)c_t-(b_1+\dots+b_{t+n})-(n-1)+1+c, -b_1+1,-a_1+1\}
\]
which is precisely the bound of \cite[Theorem~3.2]{CostaMiroRoig2}
\end{remark}

Finally we want to discuss the sharpness of the bound in
Theorem~\ref{thm:monadwreg}.

\begin{example}\label{ex:sharp}
Take $\PP=\PP(3,2,2,1)$ and consider the bundle $\cE$ given by the
monad
\[
\monad{\cO_\PP(-2)}{\alpha}{\cO_\PP(-1)\oplus\cO_\PP^{\oplus
  2}\oplus\cO_\PP(1)}{\alpha^\vee}{\cO_\PP(2)},
\]
where $\alpha^\vee=(x_0,x_1,x_2,x_3)$. In this case the bound given
by~\eqref{eq:monadwreg} is sharp.
\end{example}

In fact, we have $k=3$, $a_1=-2, b_1=-1, b_2=b_3=0, b_4=1$ and
$c_t=c=c_1=2$, so we get
\[
(m+1)3\geq\max\{(2)2-(0)-(3)+1+2, 1+1,2+1\}=4,
\]
so $m=1$. On the other hand we notice that $\cE$ is not wregular
(i.e.\ we cannot take $m=0$), so the bound is sharp. In fact from the
sequences
\[
0\to \cK\to \cO_\PP(-1)\oplus\cO_\PP^{\oplus
  2}\oplus\cO_\PP(1)\to\cO_\PP(2)\to 0
\]
and
\[
0\to\cO_\PP(-2)\to \cK\to E\to 0,
\]
we get $H^3(\cE(3-8))\neq 0$.


\begin{thebibliography}{99}
\bibitem{ArrondoMalaspina} {\sc E.~Arrondo, F.~Malaspina}, \emph{Cohomological
  characterization of vector bundles on Grassmannians of lines},
  J. Algebra \textbf{323} (2010), no.~4, 1098--1106.

\bibitem{BallicoMalaspina1} {\sc E.~Ballico, F.~Malaspina,} \emph{Qregularity and an
  extension of the Evans-Griffiths criterion to vector bundles on
  quadrics}, J. Pure Appl. Algebra \textbf{213} (2009), 194--202.

\bibitem{BallicoMalaspina2} {\sc E.~Ballico, F.~Malaspina,} \emph{Regularity and
  cohomological splitting conditions for vector bundles on
  multiprojective spaces}, J. Algebra \textbf{345} (2011), 137--149.

  \bibitem{BH}{\sc W.~Barth, K.~Hulek},
\emph{Monads and moduli of vector bundles}, Manuscripta Math. {\textbf
  25} (1978), 323--447.

\bibitem{BeltramettiRobbiano} {\sc M.~Beltrametti and L.~Robbiano,} \emph{Introduction
  to the theory of weighted projective spaces},
  Exposition. Math. \textbf{4} (1986), 111--162.

\bibitem{Canonaco} {\sc A.~Canonaco,} \emph{The Beilinson complex and
  canonical rings of irregular surfaces},
  Mem. Amer. Math. Soc. \textbf{183} (No. 862) (2006).

\bibitem{CostaMiroRoig1} {\sc L.~Costa and R.~M.~Mir\'{o}-Roig,} \emph{$m$-blocks
  collections and Cas{-}telnuovo-Mumford regularity in multiprojective
  spaces}, Nagoya Math. J. \textbf{186} (2007), 119--155.

\bibitem{CostaMiroRoig2} {\sc L.~Costa and R. M.~Mir\'o-Roig,} \emph{Monads and
  regularity of vector bundles on projective varieties}, Michigan
  Math. J. \textbf{55} (2007), 417--436.

\bibitem{FantechiMannNironi} {\sc B.~Fantechi, E.~Mann and F.~Nironi,} \emph{Smooth
  toric Deligne-Mumford stacks}, J. Reine Angew. Math. \textbf{648}
  (2010), 201--244.

\bibitem{Horrocks} {\sc G.~Horrocks}, \emph{Vector bundles on the punctured
  spectrum of a local ring}, Proc. London Math. Soc. (3) \textbf{14}
  (1964), 689--713.

\bibitem{Lazarsfeld} {\sc R.~Lazarsfeld,} \emph{Positivity in algebraic
  geometry I}, Ergebnisse der Mathematik und ihrer Grenzgebiete,
  3.~Folge, \textbf{48}. Springer-Verlag, Berlin, 2004.

\bibitem{MaclaganSmith} {\sc D.~Maclagan and G.~Smith}, \emph{Multigraded
  Castelnuovo-Mumford regularity}, J. Reine Angew. Math. \textbf{571}
  (2004), 179--212.

\bibitem{Mumford} {\sc D.~Mumford}, \emph{Lectures on curves on an algebraic
  surface}, Annals of Mathematics Studies \textbf{59}. Princeton
  University Press, Princeton, N.J., 1966.

\bibitem{PopescuPopa} {\sc G.~Pareschi and M.~Popa}, \emph{Regularity on
  abelian varieties I}, J. Amer. Math. Soc. \textbf{16} (2003),
  285--302.

\bibitem{SidmanVanTuyl}{\sc J.~Sidman and A.~Van Tuyl}, \emph{Multigraded
  regularity: syzygies and fat points}, Beitr\"{a}ge Algebra
  Geom. \textbf{47} (2006), 67--87.

\end{thebibliography}
\end{document}